\documentclass[10pt]{article}
\font\smallit=cmti10
\font\smalltt=cmtt10

\usepackage{amssymb,latexsym,amsmath,epsfig,amsthm,url,hyperref} 
\usepackage[utf8]{inputenc}
\usepackage[T1]{fontenc}
\hypersetup{
    colorlinks=true,
    linkcolor=blue,
    filecolor=magenta,      
    urlcolor=blue,
    citecolor=blue,
    pdfpagemode=FullScreen,
    }

\makeatletter

\renewcommand\section{\@startsection {section}{1}{\z@}
{-30pt \@plus -1ex \@minus -.2ex}
{2.3ex \@plus.2ex}
{\normalfont\normalsize\bfseries\boldmath}}

\renewcommand\subsection{\@startsection{subsection}{2}{\z@}
{-3.25ex\@plus -1ex \@minus -.2ex}
{1.5ex \@plus .2ex}
{\normalfont\normalsize\bfseries\boldmath}}

\renewcommand{\@seccntformat}[1]{\csname the#1\endcsname. }

\makeatother

\newtheorem{theorem}{Theorem}[section]
\newtheorem*{theorem*}{Theorem}
\newtheorem{lemma}[theorem]{Lemma}
\newtheorem{corollary}[theorem]{Corollary}

\theoremstyle{definition}
\newtheorem{definition}[theorem]{Definition}
\newtheorem{remark}[theorem]{Remark}
\newtheorem{fact}[theorem]{Fact}

\begin{document}

\begin{center}
\uppercase{\bf Nilpotent polynomials over $\mathbb{Z}$}
\vskip 20pt
{\bf Sayak Sengupta}\\
{\smallit Department of Mathematical Sciences,
		Adamas University,\\
		Barasat, Kolkata, West Bengal, India}\\
{\tt sayak.sengupta1@adamasuniversity.ac.in}\\ 
\vskip 10pt

\end{center}
\vskip 20pt
\vskip 30pt

\centerline{\bf Abstract}
\noindent
For a polynomial $u(x)$ in $\mathbb{Z}[x]$ and $r\in\mathbb{Z}$, we consider the orbit of $u$ at $r$ denoted and defined by $\mathcal{O}_u(r):=\{u^{(n)}(r)~|~n\in\mathbb{N}\}$. Here we study polynomials for which $0$ is in the orbit, and we call such polynomials \textit{nilpotent at }$r$ of index $m$ where $m=\text{min}\{n\in\mathbb N~|~u^{(n)}(r)=0\}$. We provide here a complete classification of these polynomials when $|r|\le 4$, with $|r|\le 1$ already covered in the author's previous paper, titled \textit{`Locally nilpotent polynomials over $\mathbb Z$'}. The central goal of this paper is to study the following questions: (i) relation between the integers $r$ and $m$ when the set of nilpotent polynomials at $r$ of index $m$ is non-empty, (ii) classification of the integer polynomials of index $|r|$ for large enough $|r|$, and (iii) bounded polynomial integer sequences $\{r_n\}_{n\ge 0}$.

\pagestyle{myheadings}
\markright{\smalltt Nilpotent polynomials (2024)\hfill}
\thispagestyle{empty}
\baselineskip=12.875pt
\vskip 30pt

\noindent \textbf{MSC classes}:	11A41, 37P05 (Primary), 11A05, 11A07, 37P25 (Secondary)

\section{Introduction} Polynomial iterations have been a topic of interest among mathematicians (particularly number theorists) for a long time. In this manuscript the author classifies polynomials that has a prescribed element in their orbits. Recall that given a polynomial $u\in\mathbb{Z}[x]$ and an integer $r$, the orbit of $u$ at $r$ is denoted by $\mathcal{O}_u(r)$ and defined by $$\mathcal{O}_u(r):=\{n\in\mathbb{N}~|~u^{(n)}(r)\}.$$
Throughout this paper $\mathbb{N}$ is the set of all positive integers, and for an integer polynomial $u$, we define $u^{(0)}(x):=x$ and $$u^{(n)}(x):=\underbrace{(u\circ u\circ \cdots\circ u)}_{n\text{ times}}(x),~n\in\mathbb{N}.$$ Unless specifically mentioned otherwise, by a $``$polynomial $u"$ we will mean an ``integer polynomial $u"$.

It should be noted that given any $r$ to determine all $u$'s with some $a\in \mathcal{O}_u(r)$, it is enough to consider the case when $a=0$, as whenever $a\in \mathcal{O}_u(r)$ one has $0\in \mathcal{O}_{f\circ u^{(n)}\circ f^{-1}}(r-a)$. 

\subsection{Definitions, Notation, and Terminology}

\begin{definition}
    We say that an integer sequence $(r,r_1\ldots,r_n)$ is a \textit{finite polynomial integer sequence} if it only consists of finitely many terms and there is a polynomial $u(x)$ over $\mathbb{Z}$ such that \[u(r_0)=r_1,\ldots, u(r_{n-1})=r_n.\] Throughout this paper, unless specified otherwise, by an integer sequence we will refer to a finite polynomial integer sequence. In particular, if $r_n=0$ we will say that $(r,r_1,\ldots,r_n)$ is a \textit{nilpotent integer sequence} at $r$. Moreover, if \[r_n=0 \text{ and } r_m\neq 0 \text{ for }m<n\] we say that $(r,r_1,\ldots,r_n)$ is a \textit{nilpotent integer sequence of length} $n$. 

    We fix the following notation:
    \begin{itemize}
        \item $N_{r,m}:=\text{Set of all nilpotent integer sequences at }r \text{ of length }m$.

        \item $N_r:=\cup_{m=1}^\infty N_{r,m}~.$
    \end{itemize}  
\end{definition}

In \cite{SS23}, given an integer $r$, we considered the polynomials $u$ for which $0$ is not in the orbit of $u$ at $r$, but modulo every prime $p$, there is some $m_p\in\mathbb{N}$ such that $p|u^{(m_p)}(r)$. It turns out that these polynomials can only be linear, and they have been completely classified in \cite{SS23}. In this paper we focus only on nilpotent integer sequences at positive integers $r$. One can obtain the same for negative integers $r$ using Fact 2.4 (see Section 2). In \cite{SS23} such sequences were classified for $r=0$ and $r=\pm 1$, which were obtained as part of the main results there. In this paper, we provide classification of nilpotent integer sequences at $r=2,3,$ and $4$, which can be found in Theorems 3.2, 3.4, and 3.5, respectively. We stopped calculating at $r=4$ as these were enough to prove the key results of Section 4. The classification of such sequences for arbitrary $r$'s is perhaps impossible. However, we do prove some partial and qualitative results, and we provide an inductive technique through which one can identify $N_r$ ($> 1$) whenever $N_1,\ldots,N_{r-1}$ are known. 

We also prove that the largest $m$, say $m_{\text{max}}$, for which $N_{r,m}$ is non-empty is $$m_{\text{max}}(r)=\left\{ \begin{array}{l} 2, \ \text{if} \ r=0\\
3, \ \text{if} \ |r| =1\\
4, \ \text{if} \ |r| =2\\
|r|, \ \text{if} ~|r| \geq 3 \end{array} \right.$$ We also provide a complete classification of bounded (full) integer sequences (see Corollary 4.11).

The rest of the notation that we use in this paper are as follows. For a given integer $r\ge 2$ and an integer polynomial $u$, we define 
\begin{equation}
    C_r:=\text{max}\{s\in\mathbb{N}\cup\{0\}~|~r\ge s!-s-1\}, 
\end{equation} 

\begin{equation}
  \text{and }  u_i(r):=u^{(i+1)}(r)-u^{(i)}(r),~i\in\mathbb{N}\cup\{0\}.
\end{equation}

Using the fact that for any two distinct integers $a$ \& $b$ and an integer polynomial $u$, $a-b$ divides $u(a)-u(b)$, one obtains that $u_i(r)|u_j(r)$ for every $i<j$. The usefulness of $C_r$ will be apparent in Lemma 2.1, and it plays a crucial role in the technique of most of the proofs in this paper.

The next three facts follow directly from Theorem 4.1, Corollary 4.2, and Theorem 4.4, respectively, of \cite{SS23}. These facts will be used extensively throughout the paper, and so they are reproduced here for the reader's convenience.

\begin{fact}[cf. \cite{SS23}, Theorem 4.1]
    The set $N_1$ is the union of the sets $N_{1,1}, N_{1,2}$ and $N_{1,3}$ where \[N_{1,1}=\{(1,0)\},~ N_{1,2}=\{(1,2,0)\},~ N_{1,3}=\{(1,2,3,0)\}.\]
\end{fact}

\begin{fact}[cf. \cite{SS23}, Corollary 4.2]
    The set $N_{-1}$ is the union of the sets $N_{-1,1}, N_{-1,2}$ and $N_{-1,3}$ where \[N_{-1,1}=\{(-1,0)\},~ N_{-1,2}=\{(-1,-2,0)\},~ N_{-1,3}=\{(-1,-2,-3,0)\}.\]
\end{fact}

\begin{fact}[cf. \cite{SS23}, Theorem 4.4]
    The set $N_{0}$ is the union of the sets $N_{0,1}$ and $N_{0,2}$  where \[N_{0,1}=\{(0,0)\},~ N_{0,2}=\{(0,a,0)~|~a\in\mathbb{Z}\setminus\{0\}\}.\]
\end{fact}

Our main results are the following three.

\begin{theorem}[cf. Theorems 3.2 and 3.4, and Corollaries 3.5 and 4.2]
    For any integer $r$, and $u$ nilpotent at $r$, the index of $u$ at $r$ is at most $|r|+2$. Moreover, if $|r|\ge 3$, then the index of $u$ is at most $|r|$.
\end{theorem}
    
\begin{theorem}[cf. Theorem 4.3 and Corollary 4.4]
   For any integer $r$ with $|r|\ge 5$, and $u$ nilpotent at $r$ of index $r$, $u(x)$ must be of the form $$(x-\varepsilon )+(x-\varepsilon)\cdots(x-\varepsilon r)p(x)$$ for some polynomial $p(x)$, where $\varepsilon=sgn(r)$.
\end{theorem}

The final main result is the classification of all recurringly nilpotent sequences, and the classification of all bounded integer polynomial sequences.

\begin{definition}
    Let $\{r_n\}_{n\ge 0}$ be an integer sequence. We say that $\{r_n\}_{n\ge 0}$ is a \textit{polynomial sequence} if it has a generating integer polynomial $u$, i.e., there is a polynomial $u(x)\in\mathbb{Z}[x]$ such that $u^{(n)}(r_0)=r_n$ for each $n\in \mathbb{N}$. In this case we also say that $u$ satisfies $\{r_n\}_{n\ge 0}$. Moreover, if $r_n=0$ for some $n$ (respectively, for infinitely many $n$'s) then we say $\{r_n\}_{n\ge 0}$ is a nilpotent (respectively, recurringly nilpotent) sequence. It should be noted that $\{r_n\}_{n\ge 0}$ is a recurringly nilpotent polynomial sequence is equivalent to saying that a generating integer polynomial $u$ of $\{r_n\}_{n\ge 0}$ has $0$ as a periodic point.
\end{definition}

To prove Theorems 1.5 and 1.6 we use induction on $r$, and $r=4$ serves as the base case for the induction. Thus it was necessary to include the classifications of nilpotent polynomials for $r=3\text{ and }4$, $r\in\{0,\pm 1\}$ being already available (from Facts 1.2, 1.3, and 1.4). For the proofs of 
Theorems 3.2, 3.4, and 3.5, in addition to using Facts 1.2, 1.3, and 1.4, we have also used Lemma 2.1, which is stated and proved in Section 2. 

This paper consists of 4 sections. Sections 2 and 3 are dedicated to developing the background for stating and proving the main results. Section 4 consists of the main results, and it is divided into two subsections; in Section 4.1 we study the relation between $r$ and $m$ for a polynomial $u$ in $N_{r,m}$, and in Section 4.2 we study polynomial integer sequences. This work was inspired by the previous work of the author in \cite{SS23} (where all the definitions and notation were introduced for the first time), and that in turn was inspired by the work of A. Borisov in \cite{B13}. A. Borisov noticed that all the methods contained in this manuscript would essentially still work for all functions $f:\mathbb{Z}\to \mathbb{Z}$ satisfying certain conditions, known as the \textit{locally integer polynomial functions}. See \cite{B25} for details.

Interested reader should look at the works of W. Narkiewicz and R. Marszalek in \cite{MN06}, and W. Narkiewicz in \cite{N89}, \cite{N95} and \cite{N02}. Their work involves the study of polynomial cycles in some special algebraic number fields. Other interesting works based on polynomial cycles can be found in \cite{HK02}, \cite{HN00},  \cite{P13}, \cite{P14} and \cite{Z96}. Here we are specifically interested in studying finite nilpotent integer polynomial sequences and recurringly nilpotent polynomial sequences. Section 4.2 is highly motivated and influenced by their works. Readers might be interested to look at work of Medina--Moll--Rowland in \cite{MMR15} and R.W.K Odoni in \cite{O85}.

We end this section with a few examples of nilpotent, and non-nilpotent polynomials.

\subsection{Some Examples}
\begin{itemize}
    \item Let $r\in\mathbb{Z}$. For each $q(x)\in\mathbb{Z}[x]$, $(x-r)q(x)\in N_{r,1}$.

    \item Let $r\in\mathbb Z$. For each positive integer $m$ dividing $r$, the polynomial $u_m(x)=x-r/m$ is nilpotent of index $m$.

    \item Let $u(x)=-x^3+9x^2-25x+25$. Then $u(2)=3,u(3)=4,u(4)=5, \text{ and }u(5)=0$, i.e., $u\in N_{2,4}^3$.

    \item Let $u(x)=x^3-6x^2+12x-7$. Then $u(3)=2,u(2)=1,\text{ and }u(1)=0$, i.e., $u\in N_{3,3}$.

    \item Let $u(x)=x^2-4x$. Then $u(3)=-3$, and $u(-3)=21$. As $u(x)-x>0$  on $(5,\infty)$, it follows that $0\not\in \mathcal{O}_u(3)$, i.e., $u\not\in N_3$.
\end{itemize}

\section{The Main Tools}

We begin this section with a lemma which, albeit simple, has a deep impact on the techniques of most of the proofs of this paper. 

\begin{lemma}
    Let $r,m$ be positive integers, $r\neq 1$ and $u\in N_{r,m}$ such that 
    \begin{equation}
        u(r)=r+1,\ldots,u(r+k-1)=r+k, \text{ however, } u(r+k)\neq r+k+1,
    \end{equation}
    for some non-negative integer $k$. Then $k\le C_r$, where $C_r$ is as in (1.1).
\end{lemma}

\begin{proof}
    One notices that (2.3) implies that the polynomial $u(x)-x-1$ has zeros at $r,\ldots r+k-1$, and so $u(x)$ must be of the form $$u(x)=(x+1)+(x-r)\ldots(x-r-k+1)g(x),$$ for some integer polynomial $g(x)$. Since $u(r+k)\neq r+k+1$, $g(r+k)$ must be non-zero.  Then $$u(r+k)=(r+k+1)+k!g(r+k)$$ and $g$ is a non-zero polynomial. 
    \begin{align*}
    \text{Now }  |u_k(r)|=|u^{(k+1)}(r)-u^{(k)}(r)|&=|u(r+k)-(r+k)|\\
     &=|(r+k+1)-(r+k)+k!g(r+k)|\\
     &=|1+k!g(r+k)|\\
     &\ge k!-1, \text{ where }u_i(r) \text{ is as in }(1.2).
   \end{align*}
   As $u_i(r)|u_j(r)$, whenever $i\le j$, one has $$-(r+k)=\sum\limits_{i=k}^{m-1} u_i(r)\equiv 0 \pmod {u_k(r)}.$$ This means $u_k(r)|r+k$, so that $r+k\ge k!-1$, i.e., $k\le C_r$.\end{proof}

The next lemma is an interesting consequence of Lagrange's Interpolation Theorem. The origin of this lemma goes back to at least a discussion on page 93 of \cite{MN06}.

\begin{lemma}
    Let $n\in \mathbb{N}$, and $r_0,\ldots,r_n$ be integers. Also, let $p(x)\in\mathbb{Q}[x]$ be such that $\deg(p)\le n-1$, and $p(r_i)=r_{i+1}$ for each $i\in\{0,\ldots,n-1\}$. If there is a polynomial $q(x)$ over $\mathbb{Z}$ such that $q(r_i)=p(r_i)=r_{i+1}$ for each $i\in\{0,\ldots,n-1\}$, then $p(x)\in \mathbb{Z}[x]$. Phrased differently, this means that if the interpolation polynomial for the set of points $\{(r_0,r_1),\ldots,(r_{n-1},r_n)\}$ is not in $\mathbb{Z}[x]$, then there can be no polynomial $p(x)$ with integer coefficients such that $p(r_i)=r_{i+1}$ for each $i$.
\end{lemma}

\begin{proof}
    The condition on $q$ implies that there must be some polynomial $f(x)=a_0+\cdots+a_dx^d\in \mathbb{Q}[x]$ such that \begin{equation*}
        q(x)=p(x)+(x-r_0)\cdots(x-r_{n-1})f(x).
    \end{equation*}
    One notices readily that it is enough to show that $f(x)\in\mathbb{Z}[x]$. Suppose, if possible, that $f(x)\in\mathbb{Q}[x]\setminus \mathbb{Z}[x]$. Let $m$ be the largest integer of the set $\{0\ldots,d\}$ such that $a_m$ is not in $\mathbb{Z}$. Then one has that the coefficient of $x^{n+m}$ in $(x-r_0)\cdots(x-r_{n-1})f(x)$ is not in $\mathbb{Z}$, and as $\deg p<n$, one obtains that the coefficient of $x^{n+m}$ in $(x-r_0)\cdots(x-r_{n-1})f(x)$ is equal to the coefficient of $x^{n+m}$ in $q(x)$, i.e., $q(x)\not\in\mathbb{Z}[x]$, which is absurd.
\end{proof}

The next fact will be useful in the proofs of Theorems 3.2, 3.4 and 3.5.

\begin{fact}
       Suppose that $f$ is a polynomial over $\mathbb{Z}$, $f(a/b)=c/d$, and  $\gcd(c,d)=1$. Then $P(b)\not\supset P(d)$.
\end{fact}

\begin{proof}
    Let $f(x)=\alpha_0+\alpha_1x+\cdots+\alpha_dx^d$. Then $$f(a/b)=\alpha_0+\alpha_1(a/b)+\cdots+\alpha_d(a/b)^d=\frac{\alpha_0b^d+\alpha_1ab^{d-1}+\cdots+\alpha_da^d}{b^d}$$ and the rest follows from the hypothesis.
\end{proof}

We end this section by a fact that shows a one-to-one correspondence between $N_{r,i}^d$ and $N_{-r,i}^d$.

\begin{fact}
Let $u(x)$ be an integer polynomial of degree $d$, $i\in\mathbb N$ and $r\in\mathbb{Z}\setminus\{0\}$. Define $v(x):=-u(-x)$. Then 
\begin{equation*}
    u(x)\in N_{r,i}^d \textit{ if and only if } v(x)\in N_{-r,i}^d
\end{equation*}
\end{fact}

\begin{proof}
    Since $v(-x)=-u(x)$, one obtains by induction that $v^{(n)}(-r)=-u^{(n)}(r)$, from which the fact follows. \end{proof}

\section{The Classification of Polynomials in $N_r$ for $|r|\in\{2, 3, 4\}$}

In this section we develop the necessary groundwork which is needed to prove the main results in the next section. We first make the following definition.

\begin{definition}
    Let $u\in N_{r,m}$ for some integer $r$ and some positive integer $m$. We say that $\{r,u(r),\ldots,u^{(m-1)}(r),0\}$ is the \textit{finite sequence associated to} $u$ at $r$.
\end{definition} 

We now classify the polynomials of $N_2$.

\begin{theorem} The following is the list of all polynomials in $N_2$:
\begin{enumerate}

    \item[(1)] $(x-2)p(x)$, $p(x)\in\mathbb{Z}[x]$ (Nilpotent of index 1 with the associated finite sequence $\{2,0\}$).
    
    \item[(2)] $x-1+p(x)(x-1)(x-2)$, $p(x)\in\mathbb{Z}[x]$ (Nilpotent of index 2 with the associated finite sequence $\{2,1,0\}$).

    \item[(3)] $-2(x-4)+p(x)(x-2)(x-4)$, $p(x)\in\mathbb{Z}[x]$ (Nilpotent of index 2 with the associated finite sequence $\{2,4,0\}$). 

    \item[(4)]  $-3(x-3)+p(x)(x-2)(x-3)$, $p(x)\in\mathbb{Z}[x]$ (Nilpotent of index 2 with the associated finite sequence $\{2,3,0\}$).

    \item[(5)]  $-(x-1)(x - 6)+p(x)(x-2)(x-4)(x-6)$, $p(x)\in\mathbb{Z}[x]$ (Nilpotent of index 3 with the associated finite sequence $\{2,4,6,0\}$).

    \item[(6)] $-(x-5)(x^2-4x+5)+p(x)(x-2)(x-3)(x-4)(x-5)$, $p(x)\in\mathbb{Z}[x]$ (Nilpotent of index 4 with the associated finite sequence $\{2,3,4,5,0\}$).
\end{enumerate}
    
\end{theorem}

\begin{proof}  Let $u$ be a non-zero polynomial in $N_{2,m}$ for some $m\in\mathbb{N}$. That means for any $k$ satisfying (2.3) one obtains by Lemma 2.1 that $0\le k\le C_2=3$. Also, let $u_i(2)$'s be as in (1.2).\\

\noindent{\tt Case 1.} Let $k=0$. That means $u(2)\neq 3$, i.e., $u_0(2)\neq 1$. Since $-2=\sum\limits_{i=0}^{m-1} u_i(2)\equiv 0 \pmod {u_0(2)}$, one has $u_0(2)|2$. Then $u_0(2)\in\{-1,\pm 2\}$.

 If $u_0(2)=-1$, then $u(2)=1$ and so it follows from Fact 1.2(a) that $u$ must be of the form $(x-1)p(x)$, for some polynomial $p(x)$ with $p(2)=1$.

 If $u_0(2)=-2$, then $u(2)=0$, i.e., $u(x)=(x-2)p(x)$, for some non-zero polynomial $p(x)$. 
 
If $u_0(2)=2$, then $u(2)=4$. Let $v$ be the reduction of $u$ from $2$ to $1$. Then $v\in N_{1,m}$ and $v(1)=2$. Then it follows from Fact 1.2 that $v$ must be of one of the following forms:

\begin{itemize}
    \item $v(x)=-2x+4+p(x)(x-1)(x-2)$ for some polynomial $p(x)$ (see Fact 1.2(b)), in which case $u(x)=2v\left(\frac{x}{2}\right)=-2x+8+\frac{1}{2}p\left(\frac{x}{2}\right)(x-2)(x-4)$ with the condition that $\frac{1}{2}p\left(\frac{x}{2}\right)\in \mathbb{Z}[x]$ (see Theorem 3.2(3)).

    \item $v(x)=-2x^2+7x-3+p(x)(x-1)(x-2)(x-3)$ for some polynomial $p(x)$ (see Fact 1.2(c)), in which case $u(x)=2v\left(\frac{x}{2}\right)=-x^2+7x-6+\frac{1}{4}p\left(\frac{x}{2}\right)(x-2)(x-4)(x-6)$, with the condition that $\frac{1}{4}p\left(\frac{x}{2}\right)\in \mathbb{Z}[x]$ (see Theorem 3.2(5)).
\end{itemize}

\noindent{\tt Case 2.} Let $k=1$. That means $u(2)=3$ and $u(3)\neq 4$. Since $-3=\sum\limits_{i=1}^{m-1}u_i(2)\equiv 0 \pmod{u_1(2)}$, one has $u_1(2)|3$. As $u$ is nilpotent at 2 and $u(3)\neq 4$, one immediately sees that $u_1(2)=\{\pm 3\}$.

If $u_1(2)=-3$, then $u(3)=0$, and so $u(x)$ is of the form $(x-3)p(x)$, for some $p(x)\in\mathbb{Z}[x]$ with $p(2)=-3$.

If $u_1(2)=3$, then $u(3)=6$. Let $v$ be the reduction of $u$ from $3$ to $1$. Then 
 $v\in N_{1,m}$, $v\left(\frac{2}{3}\right)=1$, and $v(1)=2$. From Fact 1.2 it follows that $v$ must be one of the two following forms:

 \begin{itemize}
     \item $v(x)=-2x+4+p(x)(x-1)(x-2)$ for some suitable $p(x)\in\mathbb{Z}[x]$. Then $1=v\left(\frac{2}{3}\right)=\frac{8}{3}+\frac{4}{9}p\left(\frac{2}{3}\right),$ i.e., $p\left(\frac{2}{3}\right)=-\frac{15}{4}$, which is impossible by Fact 2.2.

     \item $v(x)=-2x^2+7x-3+p(x)(x-1)(x-2)(x-3)$ for some suitable $p(x)\in\mathbb{Z}[x]$. Then $1=v\left(\frac{2}{3}\right)=\frac{7}{9}-\frac{28}{27}p\left(\frac{2}{3}\right),$ i.e., $p\left(\frac{2}{3}\right)=-\frac{3}{14}$, which is impossible by Fact 2.2.
 \end{itemize}

\noindent{\tt Case 3.} Let $k=2$. That means $u(2)=3,~u(3)=4$, and $u(4)\neq 3,5$. Then $u_2(2)|4$ with $u_2(2)\neq\pm 1$, i.e., $u_2(2)\in\{2,\pm 4\}$, i.e., $u(4)\in\{0,6,8\}$, i.e., in particular, $u(0)$ is even, which is false as $u(2)=3$.\\

\noindent{\tt Case 4.} Let $k=3$. That means $u(2)=3,~u(3)=4,~u(4)=5,$ and $u(5)\neq 4,6$. Then $u_3(2)|5$ with $u_3(2)\neq \pm 1$, and one sees that $u_3(2)\in\{\pm 5\}$.

If $u_3(2)=-5$, then $u(5)=0$, i.e., $u(x)=(x-5)p(x)$, for some $p(x)\in\mathbb{Z}[x]$, with $p(2)=-1,~p(3)=-2$ and $p(4)=-5$.

If $u_3(2)=5$, then $u(5)=10$. Let $v$ be the reduction of $u$ from $5$ to $1$. Then $v\in N_{1,m}$, $v\left(\frac{2}{5}\right)=\frac{3}{5}, v\left(\frac{3}{5}\right)=\frac{4}{5}$, and $v(1)=2$. Similar to Case 2 above, one can check that no such $v$ having integer coefficients can exist. \end{proof}

We will use the classification of polynomials in $N_{-2}$ later in Theorem 4.10, and so the following corollary, which follows directly from Theorem 3.2 and Fact 2.4, is recorded here for convenience.

\begin{corollary}
    The following is the list of all polynomials in $N_{-2}$:
\begin{enumerate}
    \item[(1)] $(x+2)p(x)$, $p(x)\in\mathbb{Z}[x]$ (Nilpotent of index 1 with the associated finite sequence $\{-2,0\}$).
    
    \item[(2)] $x+1+p(x)(x+1)(x+2)$, $p(x)\in\mathbb{Z}[x]$ (Nilpotent of index 2 with the associated finite sequence $\{-2,-1,0\}$).

    \item[(3)] $-2(x+4)+p(x)(x+2)(x+4)$, $p(x)\in\mathbb{Z}[x]$ (Nilpotent of index 2 with the associated finite sequence $\{-2,-4,0\}$). 

    \item[(4)]  $-3(x+3)+p(x)(x+2)(x+3)$, $p(x)\in\mathbb{Z}[x]$ (Nilpotent of index 2 with the associated finite sequence $\{-2,-3,0\}$).

    \item[(5)]  $(x+1)(x+6)+p(x)(x+2)(x+4)(x+6)$,  $p(x)\in\mathbb{Z}[x]$ (Nilpotent of index 3 with the associated finite sequence $\{-2,-4,-6,0\}$).

    \item[(6)] $-(x+5)(x^2+4x+5)+p(x)(x+2)(x+3)(x+4)(x+5)$, $p(x)\in\mathbb{Z}[x]$ (Nilpotent of index 4 with the associated finite sequence $\{-2,-3,-4,-5,0\}$).

\end{enumerate}
\end{corollary}

Next we classify the polynomials of $N_3$ using the list in Theorem 3.2.

\begin{theorem} The following is the list of all polynomials in $N_3$:
\begin{enumerate}

    \item[(1)] $(x-3)p(x)$, $p(x)\in\mathbb{Z}[x]$ (Nilpotent of index 1 with the associated finite sequence $\{3,0\}$).

    \item[(2)] $2(x-2)+p(x)(x-2)(x-3)$, $p(x)\in\mathbb{Z}[x]$ (Nilpotent of index 2 with the associated finite sequence $\{3,2,0\}$).

    \item[(3)] $-2(x-6)+p(x)(x-3)(x-6)$, $p(x)\in\mathbb{Z}[x]$ (Nilpotent of index 2 with the associated finite sequence $\{3,6,0\}$).

    \item[(4)] $-4(x-4)+p(x)(x-3)(x-4)$, $p(x)\in\mathbb{Z}[x]$ (Nilpotent of index 2 with the associated finite sequence $\{3,4,0\}$).

    \item[(5)] $x-1+p(x)(x-1)(x-2)(x-3)$, $p(x)\in\mathbb{Z}[x]$ (Nilpotent of index 3 with the associated finite sequence $\{3,2,1,0\}$).

    \item[(6)] $-(3x-13)(x-2)+p(x)(x-2)(x-3)(x-4)$, $p(x)\in\mathbb{Z}[x]$ (Nilpotent of index 3 with the associated finite sequence $\{3,4,2,0\}$).

    \item[(7)] $-(x-5)(3x-7)+p(x)(x-3)(x-4)(x-5)$, $p(x)\in\mathbb{Z}[x]$ (Nilpotent of index 3 with the associated finite sequence $\{3,4,5,0\}$).

    \item[(8)] $-2(x-4)+p(x)(x-2)(x-3)(x-4)$, $p(x)\in\mathbb{Z}[x]$ (Nilpotent of index 3 with the associated finite sequence $\{3,2,4,0\}$).
\end{enumerate}
    
\end{theorem}

\begin{proof} Let $u$ be a non-zero polynomial in $N_{3,m}$, for some $m\in\mathbb{N}$. That means for any $k$ satisfying (2.3) one obtains by Lemma 2.1 that $0\le k\le C_3=3$. Also, let $u_i(3)$'s be as in (1.2).\\

\noindent{\tt Case 1.} Let $k=0$. That means $u(3)\neq 4$, i.e., $u_0(3)\neq 1$. Since $-3=\sum\limits_{i=0}^{m-1} u_i(3)\equiv 0 \pmod{u_0(3)}$, one has $u_0(3)|3$. i.e., $u_0(3)\in\{-1,\pm 3\}$, as $u_0(3)$ cannot be 1.\\

 If $u_0(3)=-1$, then $u(3)=2$, i.e., $u\in N_{2,m-1}$. It follows from the list in Theorem 3.2 that $u$ must be of one of the following forms:\begin{itemize}
    
    \item $u(x)=(x-2)p(x)$, with $p(3)=2$ (see Theorem 3.2(1)).

    \item $u(x)=(x-1)^2+p(x)(x-1)(x-2)$, with $p(3)=-1$ (see Theorem 3.2(2)).

    \item $u(x)=-2x+8+p(x)(x-2)(x-4)$, with $p(3)=0$ (see Theorem 3.2(3)). 
 \end{itemize}

 If $u_0(3)=-3$, then $u(3)=0$, i.e., $u(x)=(x-3)p(x)$, for some non-zero polynomial $p(x)$.

 If $u_0(3)=3$, then $u(3)=6$. Let $v$ be the reduction of $u$ from $3$ to $1$. So $v(1)=2$ and $v\in N_{1,m}.$ Now it follows from Fact 1.2 that $v$ must be of one of the following forms:
\begin{itemize}
    \item $v(x)=-2x+4+p(x)(x-1)(x-2)$ for some polynomial $p(x)$ (see Fact 1.2(b)), in which case $u(x)=-2x+12+\frac{1}{3}p\left(\frac{x}{3}\right)(x-3)(x-6)$ with the condition that $\frac{1}{3}p\left(\frac{x}{3}\right)$ is a polynomial over $\mathbb{Z}$. 

    \item $v(x)=-2x^2+7x-3+p(x)(x-1)(x-2)(x-3)$ for some polynomial $p(x)$ (see Fact 1.2(c)), in which case $u(x)=-\frac{2}{3}x^2+7x-9+\frac{1}{9}p\left(\frac{x}{3}\right)(x-3)(x-6)(x-9)$. However it follows from Lemma 2.2 that this $u$ cannot be a polynomial over $\mathbb{Z}$ for any choice of an integer polynomial $p(x)$.
\end{itemize}

\noindent{\tt Case 2.} Let $k=1$. That means $u(3)=4$, and $u(4)\neq 5$. Then one has $u_1(3)|4$ with $u_1(3)\neq\pm 1$, i.e., $u_1(3)\in\{\pm 2,\pm 4\}$. (Note that $u_1(3)$ cannot be $\pm 1$, as that would imply $u(4)=3$ in which case it cannot be nilpotent, or $u(4)=5$, which is not possible in Case 2.) If $u_1(3)=-2$, then $u(4)=2$, i.e., $u\in N_{2,m-1}$. From Theorem 3.2 one sees that $u$ must be of the form $(x-2)p(x)$, with $p(3)=4,p(4)=1$.

 If $u_1(3)=2$, then $u(4)=6$. Let $v$ be the reduction of $u$ from $4$ to $2$. Then $v\left(\frac{3}{2}\right)=2, v(2)=3$, and $v\in N_{2,m-1}$. Using the list in Theorem 3.2 and Fact 2.3 one can check that (similar to what was done in Case 2 of Theorem 3.2) no such polynomial $v$ with integer coefficients can exist. 

 If $u_1(3)=-4$, then $u(4)=0,$ i.e., $u(x)=(x-4)p(x)$ with $p(3)=-4$.

 If $u_1(3)=4,$ then $u(4)=8$. Let $v$ be the reduction of $u$ from $4$ to $1$. Then $v\in N_{1,m-1}, v\left(\frac{3}{4}\right)=1$, and $v(1)=2$. Using Fact 1.2 and Fact 2.3 one can check that (similar to what was done in Case 2 of Theorem 3.2) no such $v$ is possible with integer coefficients.\\

\noindent{\tt Case 3.} Let $k=2$. That means $u(3)=4,~u(4)=5$, but $u(5)\neq 4, 6$. Then one has $u_2(3)|5$ with $u_2(3)\neq 1$, i.e., $u_2(3)\in\{\pm 1,\pm 5\}$.

 If $u_2(3)=-5$, then $u(5)=0$, i.e., $u(x)=(x-5)p(x)$ with $p(3)=-2,p(4)=-5$.

 If $u_2(3)=5$, then $u(5)=10$. Let $v$ be the reduction of $u$ from $5$ to $1$. Then $v\left(\frac{3}{5}\right)=\frac{4}{5},v\left(\frac{4}{5}\right)=1, v(1)=2$ and $v\in N_{1,m-2}$. Using Fact 1.2 and Fact 2.3 one can check that (similar to what was done in Case 2 of Theorem 3.2) no such $v$ is possible with integer coefficients.\\

\noindent{\tt Case 4.} Let $k=3$. That means $u(3)=4, u(4)=5, u(5)=6$, but $u(6)\neq 7$. Then $u_3(3)|6$, i.e., $u_3(3)\in\{\pm 1,\pm 2,\pm 3,\pm 6\}$. Note that $u_3(3)$ cannot be $\pm 1$, $-2$ or $-3$. The rest of the cases can be checked using similar arguments as above, and one shall see that no new polynomials arise here. \end{proof}

Finally we are ready to classify the polynomials of $N_4$.

\begin{theorem} The following is the list of all polynomials in $N_4$:
\begin{enumerate}

    \item[(1)] $(x-4)p(x)$, $p(x)\in\mathbb{Z}[x]$ (Nilpotent of index 1 with the associated finite sequence $\{4,0\}$).
    
    \item[(2)] $3(x-3)+p(x)(x-3)(x-4)$, $p(x)\in\mathbb{Z}[x]$ (Nilpotent of index 2 with the associated finite sequence $\{4,3,0\}$).

    \item[(3)] $-3(x-6)+p(x)(x-4)(x-6)$, $p(x)\in\mathbb{Z}[x]$ (Nilpotent of index 2 with the associated finite sequence $\{4,6,0\}$).

    \item[(4)] $x-2+p(x)(x-2)(x-4)$, $p(x)\in\mathbb{Z}[x]$ (Nilpotent of index 2 with the associated finite sequence $\{4,2,0\}$).

    \item[(5)] $-2(x-8)+p(x)(x-4)(x-8)$, $p(x)\in\mathbb{Z}[x]$ (Nilpotent of index 2 with the associated finite sequence $\{4,8,0\}$).

    \item[(6)] $-5(x-5)+p(x)(x-4)(x-5)$, $p(x)\in\mathbb{Z}[x]$ (Nilpotent of index 2 with the associated finite sequence $\{4,5,0\}$). 

    \item[(7)] $x-1+p(x)(x-1)(x-2)(x-3)(x-4)$, $p(x)\in\mathbb{Z}[x]$ (Nilpotent of index 4 with the associated finite sequence $\{4,3,2,1,0\}$).

    \item[(8)] $-(2x-13)(x-3)+p(x)(x-3)(x-4)(x-5)(x-6),$ $p(x)\in\mathbb{Z}[x]$ (Nilpotent of index 4 with the associated finite sequence $\{4,5,6,3,0\}$).
\end{enumerate}

\end{theorem}

\begin{proof} Let $u$ be a non-zero polynomial in $N_{4,m}$, for some $m\in\mathbb{N}$. That means for any $k$ satisfying (2.3) one obtains by Lemma 2.1 that $0\le k\le C_4=3$. Also, let $u_i(4)$'s be as in (1.2).\\

\noindent{\tt Case 1.} Let $k=0$. That means $u(4)\neq 5$. Since $-4=\sum\limits_{i=0}^{m-1} u_i(4)\equiv 0 \pmod{u_0(4)}$, one has $u_0(4)|4$. i.e., $u_0(4)\in\{-1,\pm 2,\pm 4\}$, as $u_0(4)$ cannot be 1.

 If $u_0(4)=-1$, then $u(4)=3$, i.e., $u\in N_{3,m-1}$. It follows from the list in Theorem 3.4 that $u$ must be of one of the following forms:\begin{itemize}
    \item $u(x)=(x-3)p(x)$ for some polynomial $p(x)$ with $p(4)=3$ (see Theorem 3.4(1)).

    \item $u(x)=(x-1)+p(x)(x-1)(x-2)(x-3)(x-4)$ for some polynomial $p(x)$ (see Theorem 3.4(2)).
\end{itemize}

 If $u_0(4)=2$, then $u(4)=6$. Let $v$ be the reduction of $u$ from $4$ to $2$. Then $v(2)=3$ and $v\in N_{2,m}$. Looking at the list in Theorem 3.2 one can see that $v$ must be of one of the following forms:
\begin{itemize}
    \item $v(x)=-3x+9+p(x)(x-2)(x-3)$ for some polynomial $p(x)$. Then $u(x)=2v\left(\frac{x}{2}\right)=-3x+18+\frac{1}{2}p\left(\frac{x}{2}\right)(x-4)(x-6)$.

    \item $v(x)=-x^3+9x^2-25x+25+p(x)(x-2)(x-3)(x-4)(x-5)$ for some polynomial $p(x)$. However, then one has $u(x)=2v\left(\frac{x}{2}\right)=-\frac{1}{4}x^3+\frac{9}{2}x^2-25x+50+\frac{1}{8}p\left(\frac{x}{2}\right)(x-4)(x-6)(x-8)(x-10)$, which, by Lemma 2.2, cannot be a polynomial over $\mathbb{Z}$ for any choice of integer polynomial $p(x)$.
\end{itemize}

 If $u_0(4)=-2$, then $u(4)=2$. Let $v$ be the reduction of $u$ from $4$ to $2$. Then $v(2)=1$ and $v\in N_{2,m}.$ From Theorem 3.2 it follows that $v$ must be of the form $(x-1)+p(x)(x-1)(x-2)$ for some polynomial $p(x)$ (see Theorem 3.2(2)), and then $u(x)=2v\left(\frac{x}{2}\right)=x-2+\frac{1}{2}p\left(\frac{x}{2}\right)(x-2)(x-4)$.

 If $u_0(4)=4$, then $u(4)=8$. Let $v$ be the reduction of $u$ from $4$ to $1$. Then $v(1)=2$ and $v\in N_{1,m}.$ From the list in Fact 1.2 it follows that $v$ must be of one of the following forms:
\begin{itemize}
    \item $v(x)=-2x+4+p(x)(x-1)(x-2)$, i.e., $u(x)=-2x+16+\frac{1}{4}p\left(\frac{x}{4}\right)(x-4)(x-8)$ (see Fact 1.2(b)). 

    \item $v(x)=-2x^2+7x-3+p(x)(x-1)(x-2)(x-3)$, i.e., $u(x)=-\frac{1}{2}x^2+7x-12+\frac{1}{16}p\left(\frac{x}{4}\right)(x-4)(x-8)(x-12)$ (see Fact 1.2(c)). However, this cannot be a polynomial over $\mathbb{Z}$ for any choice of an integer polynomial $p(x)$, which follows from Lemma 2.2.
\end{itemize}

If $u_0(4)=-4$, then $u(4)=0$, i.e., $u(x)=(x-4)p(x)$, for some non-zero $p(x)$.\\

\noindent{\tt Case 2.} Let $k=1$. That means $u(4)=5,~u(5)\neq 6$. Then one has $u_1(4)|5$, i.e., $u_1(4)\in\{\pm 5\}$. (Note that $u_1(4)$ cannot be $\pm 1$, as that would imply $u(5)=4$, in which case it cannot be nilpotent, or $u(5)=6$, which is not possible in this case.)

 If $u_1(4)=-5$, then $u(5)=0$, i.e., $u(x)=(x-5)p(x)$, with $p(4)=-5$.

 If $u_1(4)=5$, then $u(5)=10$. Let $v$ be the reduction of $u$ from $5$ to $1$. Then $v(1)=2,v\left(\frac{4}{5}\right)=1$, and $v\in N_{1,m-1}.$ One can use list in Fact 1.2 and Fact 2.3 to check that (similar to what was done in Case 2 of Theorem 3.2) such a polynomial $v$ with integer coefficients cannot exist.\\

\noindent{\tt Case 3.} Let $k=2$. That means $u(4)=5,~u(5)=6$, but $u(6)\neq 7$. Then one has $u_2(4)|6$ $u_2(4)\neq \pm 1$, i.e., $u_2(4)\in\{\pm 2,\pm 3,\pm 6\}$. Also, $u_2(4)$ cannot be even, as that would imply $u(6)$ is even, i.e., $u(0)$ is even, and that cannot happen as $u(4)$ is odd. Thus $u_2(4)\not\in\{\pm 2,\pm 6\}$. That means $u_2(3)\in\{\pm 3\}$.

 If $u_2(4)=-3$, then $u(6)=3$. Let $v$ be the reduction of $u$ from $6$ to $2$. Then $v\left(\frac{4}{3}\right)=\frac{5}{3},v\left(\frac{5}{3}\right)=2, v(2)=1$, and $v\in N_{2,m-2}$. It follows from Theorem 3.2 that $v(1)=0$, i.e., $u(3)=0$. Thus $u(x)=-2x^2+19x-39+p(x)(x-3)(x-4)(x-5)(x-6),$ for some $p(x)\in\mathbb{Z}[x]$.

If $u_2(4)=3$, then $u(6)=9$. Let $v$ be the reduction of $u$ from $6$ to $2$. Then $v\left(\frac{4}{3}\right)=\frac{5}{3},v\left(\frac{5}{3}\right)=2, v(2)=3$ and $v\in N_{2,m-2}$. Using Theorem 3.2 and Fact 2.3 one can check that (similar to what was done in Case 2 of Theorem 3.2) no such $v$ with integer coefficients can exist.\\

\noindent{\tt Case 4.} Let $k=3$. That means $u(4)=5, u(5)=6, u(6)=7$, but $u(7)\neq 8$. Using similar arguments as in Case 3, we can reject all the possibilities that arise here. \end{proof}

The classification of the polynomials in $N_{-3}$ and $N_{-4}$ follows from Theorems 3.4, 3.5, and Fact 2.4. As we do not need to know such classifications for our next section that contains the main results we skip the part of writing them out.

\section{The Main Results}
We state and prove the main results of this paper in this section.

\subsection{Relation Between the Integers $r$ and $m$ in $N_{r,m}$ }

We start with the first main result.

\begin{theorem}
 Let $r$ be an integer greater than 3 and $u$ be an integer polynomial such that $u\in N_{r,m}$. Then $m\le r$, i.e., in other words, $N_r=\sqcup_{i=1}^{r} N_{r,i}$.
\end{theorem}

\begin{proof}
    We will use induction on $r$. The base case follows from Theorem 3.5. Let $r$ be an integer greater than or equal to 4, and assume that the statement is true for every integer $q$ satisfying $4\le q\le r$. We want to prove that the statement is true for $r+1$, i.e., if $u\in N_{r+1,m}$, then $r+1\ge m$. Then by Lemma 2.1, the $k$, for which (2.3) holds, must be between 0 and $C_{r+1}$, where $C_{r+1}$ is as in (1.1). Let $u_i(r+1)$'s be as in (1.2). Define \begin{align*}
        a&:=|u_k(r+1)| \text{ and}\\
        b&:=\frac{r+k+1}{a},
    \end{align*}
    so that $ab=r+k+1$.
    
    First suppose that $k=0$, i.e., $u(r+1)\neq r+2$. Then $-(r+1)=\sum\limits_{i=0}^{m-1} u_i(r+1)\equiv 0 \pmod {u_0(r+1)},$ i.e., $u_0(r+1)|r+1$. The $u_0(r+1)$, of course, cannot be $1$. If $u_0(r+1)=-1$, then $u(r+1)=r$, so that $u\in N_{r,m-1}$. Then by the induction hypothesis, we get $m-1\le r,$ i.e., $m\le r+1$. Now suppose that $u_0(r+1)$ is a non-trivial divisor of $r+1$. Define $v(x)=\frac{1}{a}u(ax)$. Then $v\in N_{b,m}$ and $b<r+1$. If $b\ge 4$, then $m\le b< r+1$. Otherwise $b\in\{1,2,3\}$. If $b=1$ then by Fact 1.2 one has $m\le 3<r+1$. Similarly, one treats the possibilities $b=2$, and $b=3$ using Theorems 3.2 and 3.4, respectively.

    Now, let us suppose that $1\le k\le C_{r+1}$. Then $$u(x)=(x+1)+(x-r-1)\cdots(x-r-k)g(x),$$ for some polynomial $g(x)$ with $g(r+k+1)\neq 0$ and $u_k(r+1)\neq \pm 1$, as otherwise $u(r+k+1)=r+k+2$ or $r+k$, which, in the first case, goes against the definition of $k$, and in the second case, gives a non-nilpotent polynomial at $r+1$. Also $-(r+k+1)=\sum\limits_{i=k}^{m-1} u_i(r+1)\equiv 0 \pmod {u_k(r+1)},$ i.e., $u_k(r+1)|r+k+1$. Let $p\in P(u_k(r+1))$. Then $p$ must also divide $r+k+1$ and $u(0)$. Note that 
    \begin{align*}
        a=|u_k(r+1)|&=|u(r+k+1)-(r+k+1)|\\
        &=|(r+k+2)+k!g(r+k+1)-(r+k+1)|\\
        &=|1+k!g(r+k+1)|,
    \end{align*}
    i.e., $p|1+k!g(r+k+1)$. Since $p|r+k+1$, one deduces that $p|1+k!g(0)$. So, in particular, one has $a>k$. Define $v(x)=\frac{1}{a}u(ax)$, then $v\in N_{b,m-k}$, with $b=\frac{r}{a}+\frac{k+1}{a}<r+1$. One immediately sees that if $r\ge 6$ then in particular, one has $k\le C_r \le \frac{r}{2}$. This will be useful in Case 2 below. \\

    \noindent{\tt Case 1.} Let $r=4$ and $1\le k\le C_{5}=3$. First suppose that $b=\frac{5+k}{|u_k(5)|}\in\{1,2,3\}$. Note that $k$ cannot be 3 here, as otherwise $|u_k(5)|\in\{4,8\}$ and that implies $u(5)=6,u(6)=7,u(7)=8,u(8)\in\{0,4,12,16\}$; however, $2|u(8)-u(6)\in \{-7,-3,5,9\}$, which is false. That means $k\in\{1,2\}$.  
    If $b=1$, then by Fact 1.2 applied to $v$ one obtains $m-k\le 3$, i.e., $m\le k+3\le 5=r+1$. Similarly, one treats the possibilities 
    $b=2$ and 3 using Theorems 3.2 and 3.4, respectively. 
    
    Now suppose that $4\le b $. Of course $b<r+1=5$, so that $b=4$. Then $k$ must be 3, i.e., $a=2$. That means $u(5)=6,u(6)=7,u(7)=8,u(8)=10$; however, that would mean $2|u(8)-u(6)=3$, which is false. So this cannot happen.\\

\noindent{\tt Case 2.} Let $r\ge 5$ and $1\le k\le C_{r+1}$.
   First suppose that $b\in\{1,2,3\}$. If $b=1$, then by Fact 1.2 applied to $v$ one has $m-k\le 3$, i.e., $m\le k+3\le r+1$. Similarly, one treats the possibilities 
    $b=2$ and $3$ using Theorems 3.2 and 3.4, respectively.

    Now suppose that $b\ge 4$. Then by the induction, one has $m-k\le b$, i.e., $m\le k+b\le \frac{r+1}{2}+\frac{r+k+1}{k+1}\le r+3/2$. Now $m$ being an integer it follows that $m\le r+1$. This concludes the proof. \end{proof}

The next corollary is immediate from Facts 1.2, 1.3, 1.4, 2.4, Theorem 3.2, and Corollary 3.3.

\begin{corollary}
    If $r$ is an integer with $|r|\ge 3$, and and $u$ be an integer polynomial such that $u\in N_{r,m}$ for some $m$, then one must have $m\le |r|$, and if $|r|\le 2$ then $m\le |r|+2$.
\end{corollary}

Now we prove our second main result of this paper.
    
\begin{theorem}
    If $r$ is an positive integer greater than 3 and $u\in N_{r,r}$, then one of the following is true:
    \begin{enumerate}
        \item[(i)] $r=4$, and either $u(x)$ is of the form $(x-1)+(x-1)\cdots(x-r)p(x)$, or of the form $-2x^2+19x-39+p(x)(x-3)(x-4)(x-5)(x-6)$;

        \item[(ii)] $r\ge 5$, and $u(x)$ is of the form $(x-1)+(x-1)\cdots(x-r)p(x)$;
    \end{enumerate} for some $p(x)\in \mathbb{Z}[x]$.
\end{theorem}

\begin{proof}
    We will use induction on $r$, and use the method of proof of Theorem 4.1. The base case follows from Theorem 3.5(7) and 3.5(8). Let $r$ be a positive integer greater than or equal to 4, and assume that the statement is true for every integer $q$ with $4\le q\le r$. Also suppose that $u\in N_{r+1,r+1}$. Let $u_i(r+1)$'s be as in (1.2), $k$ be an integer satisfying (2.3) for the given $u$ and $r+1$, and $a,b$ be as in the proof of Theorem 4.1. Suppose that $1\le k\le C_{r+1}$. Then one can define $v$ as in the third paragraph of the proof of Theorem 4.1, and deduce that $v\in N_{b,r+1-k}$. We also get that $a=|u_k(r+1)|>k$, and that $k\le C_r\le r/2$ when $r\ge 6$. First we claim that $k=0$. For a contradiction suppose that $k>0$. Like Theorem 4.1 we look at the following two cases.\\

    \noindent{\texttt{Case 1.}} Let $r=4$ and $1\le k\le C_5= 3$. Then $$b=\frac{r+1+k}{|u_k(r+1)|}=\frac{5+k}{|u_k(5)|}\le \frac{5+k}{1+k}=1+\frac{4}{k+1}\in\left\{2,\frac{7}{3},3\right\},$$ so that $b\in\{1,2,3\}$. One deduces that $k\neq 3$ by using similar arguments as in Case 1 of Theorem 4.1. 
    
    If $k=1$, then $u(5)=6, u(6)\neq 7$ and $b=\frac{6}{|u_1(5)|}\le \frac{6}{2}=3$, i.e., $u(6)\in\{0,3,4,8,9,12\}$. Note that $u(6)=0$ is not possible as $u\in N_{5,5}$, and $u(6)=12$ is also not possible as then $v$ must lie in $N_{1,4}$, which is an empty set by Fact 1.2. If $u(6)=3$, then $v\in N_{2,4}$. It follows from Theorem 3.2(6) that $v$ has the associated finite sequence $\{2,3,4,5,0\}$, so that $3=u(6)=3v(2)=9$, which is false. Similarly, if $u(6)=9$ then $v\in N_{2,4}$, and $u(9)=3v(3)=12$. However that is not possible as otherwise one has $4=9-5|u(9)-u(5)=12-6=6$. If $u(6)=4$ then $v\in N_{3,4}$, which is absurd as $N_{3,4}$ is an empty set by Theorem 3.4. Similar arguments can be used if $u(6)=8$. 
    
    If $k=2$, then $u(5)=6,u(6)=7, u(7)\neq 8$, and $b=\frac{7}{|u_2(5)|}\le\frac{7}{3}$, and so $u(7)\in\{0,14\}$. Of course $u(7)$ cannot be zero as $u\in N_{5,5}$. Also, if $u(7)=14$ then $v\in N_{1,4}$, which is an empty set by Fact 1.2.\\

    \noindent{\texttt{Case 2.}} Let $r\ge 5$ and $1\le k\le C_{r+1}$.

    If $b=1$, then by Fact 1.2 applied to $v$ one has $r+1-k\le 3$, i.e., $r-k\le 2$. This can only be possible if $r=5$ and $k=C_6=3$, i.e., $u(6)=7,u(7)=8,u(8)=9,u(9)\neq 10$. As $u(9)-9$ divides $9$, $u(9)\in\{0,12,18\}$. One can now check that the interpolation polynomial for each of these possibilities are outside of $\mathbb Z[x]$. The rest follows from Lemma 2.2.

    If $b=2$, then by Theorem 3.2 applied to $v$ one has $r+1-k\le 4$, i.e., $r-k\le 3$. There are only three ways that this can be true, and they are $(r,k)\in\{(5,2),(5,3),(6,3)\}$. Using similar arguments as in the previous paragraph one can reject all these three possibilities.

    Finally, if $b\ge 3$ then by Corollary 4.2 applied on $v$ one obtains
    \begin{align*}
      & r+1-k\le     b \\
      \implies & r+1 \le k+b\le \frac{r+1}{2}+\frac{r}{k+1}+1 \text{ (as } a\ge k+1\text{)}\\
      \implies & \frac{r}{2}\le \frac{1}{2}+\frac{r}{k+1}\\
      \implies & r\le \frac{k+1}{k-1}\le 3~\text{, a contradiction}.
    \end{align*}
    Thus one must have $k=0.$ The $u_0(r+1)$ cannot be $1$. If $u_0(r+1)=-1$, then $u(r+1)=r$, i.e., $u\in N_{r,r}$. So by the induction hypothesis $$u(x)=(x-1)+(x-1)\cdots (x-r)p(x),$$ with $p(r+1)=0$, and that means $p(x)$ has the linear factor $x-r-1$. Now suppose $a=|u_0(r+1)|>1$. Then one defines $v$ as in the second paragraph of the proof of Theorem 4.1, and deduce that $v\in N_{b,r+1}$. Then Corollary 4.2 applied to $v$ implies that $r+1\le b+2,$ i.e., $r\le \frac{a+1}{a-1}\le 3$, which is false, as $r\ge 4$. This completes the proof. \end{proof}

From Theorem 4.3 and Fact 2.4 it follows that
\begin{corollary}
    If $r$ is a negative integer less than $-3$ and $u\in N_{r,|r|}$, then one of the following is true:
    \begin{enumerate}
        \item[(i)] $r=-4$, and either $u(x)$ is of the form $(x+1)+(x+1)\cdots(x+r)p(x)$, or of the form $2x^2+19x+39+p(x)(x+3)(x+4)(x+5)(x+6)$;

        \item[(ii)] $|r|\ge 5$, and $u(x)$ is of the form $(x+1)+(x+1)\cdots(x+r)p(x)$;
    \end{enumerate} for some $p(x)\in \mathbb{Z}[x]$.
\end{corollary}

It is clear that considering all possible nilpotent polynomials at $r$ of index $n$ is the same as considering all possible finite nilpotent integer polynomial sequences (see Definition 4.5 below) $(r_0,r_1,\ldots,r_{n-1},0)$, where $r=r_0$ and $r_n=0$. This follows from Lemma 2.2 of the dissertation, and also because of the fact that if $L(x)$ is the Interpolation polynomial of degree less than or equal to $n-1$ with $L(r_{i-1})=r_i$, then all possible nilpotent polynomials at $r$ of index $n$ must be of the form $$L(x)+p(x)x(x-r)(x-r_1)\cdots(x-r_{n-1}).$$  

In that spirit, we dedicate the next section to the study of such sequences.

\subsection{Recurringly Nilpotent Polynomials over $\mathbb{Z}$}

We first recall some definitions from the introduction.

\begin{definition}
    We say that an integer sequence $\{r_n\}_{n\ge 0}$ is a \textit{(recurringly) nilpotent sequence} if $r_n=0$ for (infinitely many) some $n$'s.
\end{definition}

\begin{definition}
    Let $\{r_n\}_{n\ge 0}$ be an integer sequence. We say that $\{r_n\}_{n\ge 0}$ is a polynomial sequence if it has a \textit{generating integer polynomial} $u$, i.e., there is a polynomial $u(x)\in\mathbb{Z}[x]$ such that $u^{(n)}(r_0)=r_n$ for each $n\in \mathbb{N}$. In this case we also say that $u$ satisfies $\{r_n\}_{n\ge 0}$.
\end{definition}

\noindent{\textbf{Remark.}} In \cite{MN06} the authors use the term \textit{sequence realized by a polynomial} instead of a \textit{generating polynomial} of a sequence.

\noindent{\textbf{Remark.}} If $\{r_n\}_{n\ge 0}$ is a recurringly nilpotent integer polynomial sequence having a generating integer polynomial $u$, then the orbit of $u$ at $r_0$ is finite. Thus, in this case $r_0$ is a pre-periodic point of $u$ and $0$ is a periodic point of $u$.

We now state and prove a theorem that describes what recurringly nilpotent polynomial sequences look like.

\begin{theorem}
    Let $\{r_n\}_{n\ge 0}$ be a polynomial sequence, and $u$ be an integer polynomial that satisfies $\{r_n\}_{n\ge 0}$. Then the following are equivalent:
    \begin{enumerate}
        \item[(a)] The sequence $\{r_n\}_{n\ge 0}$ is a recurringly nilpotent sequence, i.e., $r_0$ is a pre-periodic point of $u$ (or $0$ is a periodic point of $u$).

        \item[(b)] There exists a positive integer $m$ such that either $r_m=r_{m+1}=0$, or $r_m=r_{m+2}=0$.
    \end{enumerate}
\end{theorem}

\begin{proof}
    The fact that $(b)$ implies $(a)$ is clear. Therefore we only prove that $(a)$ implies $(b)$. We take the two smallest positive integers $m\text{ and }q$ with $m<q$, such that $r_m=r_q=0$. That means $u$ is nilpotent at $0$. It follows from Fact 1.4 that $q\in\{m+1,m+2\}$. If $q=m+1$, then $\{r_n\}_{n\ge 0}$ is of the form
    \begin{equation}
        r_0,\ldots,r_{m-1},0,0,\ldots 
    \end{equation}
    and one has $r_n=u^{(n)}(r_0)=0$ for every $n\ge m$, and if $q=m+2$, then $\{r_n\}_{n\ge 0}$ is of the form 
    \begin{equation}
        r_0,\ldots,r_{m-1},0,r_{m+1},0,\ldots
    \end{equation}
    and one has $r_n=u^{(n)}(r_0)=0$ for every $n\in\{m+2k~|~k\ge 0\}$.
\end{proof}

\begin{remark}
   One can see that if $\{r_n\}_{n\ge 0}$ is an polynomial sequence satisfying the conditions of Theorem 4.7 then, in particular, $\{r_n\}_{n\ge 0}$ is a bounded sequence.
\end{remark}

From Theorem 4.7 it follows that all the recurringly nilpotent polynomial sequences must be of one of the forms (4.4) or (4.5). We would like to classify all such polynomial sequences. The next two theorems are dedicated to that.

\begin{theorem}
    Let $m$ be a non-negative positive integer, and $\{r_n\}_{n\ge 0}$ be a recurringly nilpotent sequence of the form $(4.4)$. We assume that $(4.4)$ is the zero sequence if $m=0$. When $m\ge 1$, we suppose that $0\not\in\{r_0,\ldots,r_{m-1}\}$.
    Also suppose that $u$ is a generating polynomial of $\{r_n\}_{n\ge 0}$. Then exactly one of the following holds:
    \begin{enumerate}
        \item[(1)] $m=0$, and $\{r_n\}_{n\ge 0}$ is the zero sequence.
    
        \item[(2)] $m=1$, and then $r_0$ is arbitrary; and

        \item[(3)] $m=2$, and then one has $r_0|2$, and $r_1=2r_0$.
    \end{enumerate}
\end{theorem}

\begin{proof}
It is enough to consider the case $m\ge 1$. The given conditions on $u$, in particular, implies that $u\in N_{0,1}$. Then by Fact 1.4(a) one obtains $u(x)=xp(x)$, for some non-zero polynomial $p$, so, in particular, $r_0|r_1|\cdots |r_{m-1}$. Define $v(x)=\frac{1}{r_0}u(r_0 x)$. Then $v$ is nilpotent at $1$ of index $m$ with $v(0)=0$. From Fact 1.2 it follows that $v$ must be of one of the following forms:
    \begin{enumerate}
        \item[(i)] $v(x)=x(x-1)Q(x)$ for some non-zero polynomial $Q$. In this case $m=1$, i.e., the above sequence is $r_0,0,0,\ldots$, and $u(x)=x(x-r_0)P(x)$, where $P(x)=\frac{1}{r_0}u\left(\frac{x}{r_0}\right)$ .

        \item[(ii)] $v(x)=-2x+4+Q(x)(x-1)(x-2)$ with $Q(0)=-2$. In this case $m=2$, i.e., the above sequence is $r_0,r_1,0,0,\ldots$, and $u(x)=-2x+4r_0+P(x)(x-r_0)(x-r_1)$, where $P(x)=\frac{1}{r_0}Q\left(\frac{x}{r_0}\right)$ with $P(0)=-\frac{2}{r_0}$, i.e., one requires $r_0|2$. One also has $r_1=u(r_0)=2r_0$.     \end{enumerate}
This completes the proof. \end{proof}

\begin{theorem}
    Let $m$ a non-negative integer, $\{r_n\}_{n\ge 0}$ a recurringly nilpotent sequence of the form $(4.5)$, and $u$ be a generating polynomial of $\{r_n\}_{n\ge 0}$. Then one of the following holds:
    \begin{enumerate}

        \item[(1)] $m=0$, and then $r_1$ is arbitrary;  
        
        \item[(2)] $m=1$, and then either $r_0=r_2$, or, $r_0=\varepsilon$;

        \item[(3)] $m=2$, and then either $r_3=r_1=r_0+\varepsilon$, or, $r_0=2\varepsilon,r_1=\varepsilon, \text{ and } r_3=3\varepsilon$; and

        \item[(4)] $m=3$, and then $r_0=\varepsilon,r_1=2\varepsilon,\text{ and }r_4=r_2=3\varepsilon$,
    \end{enumerate}
    for $\varepsilon\in\{\pm 1\}$.
\end{theorem}

\begin{proof} Let $v(x)=u^{(2)}(x)$. One readily notices that if $m$ is even (respectively, $m$ is odd), the sequence of iterations of $v$ starting at $r_0$ (respectively, starting at $r_1$) will be of the form (4.4). Then by Theorem 4.9 it follows that it is enough to look at the cases when $m\le 2\cdot 2+1=5$. Throughout this proof $\varepsilon=\pm 1$.\\

\noindent{\tt Case 1.} Let $m=0$, i.e., $r_0=0$ and $r_1\neq 0$. Then $u(x)=r_1-x$ satisfies (4.5).\\

\noindent{\tt Case 2.} Let $m=1$. Then the sequence looks like $r_0,0,r_2,0,\ldots$. From Fact 1.4(b) it follows that there is a polynomial $p$ such that $u(x)=(x-r_2)(xp(x)-1)$. That means, in particular, $$0=(r_0-r_2)(r_0p(r_0)-1),~ i.e.,~either~ r_0=r_2,~ or ~r_0=\varepsilon.$$
\begin{itemize}
    \item If $r_0=r_2$, then $u(x)=r_0-x$ satisfies (4.5).

    \item If $r_0=\varepsilon$, then $u(x)=\varepsilon(x-r_2)(x-\varepsilon)$ satisfies (4.5).
\end{itemize}

\noindent{\tt Case 3.} Let $m=2$. Then the sequence looks like $r_0,r_1,0,r_3,0,\ldots$. Then by Fact 1.4(b) it follows that $u(x)=(x-r_3)(xp(x)-1)$ for some polynomial $p$. That means, in particular, $$0=u(r_1)=(r_1-r_3)(r_1p(r_1)-1),~ i.e., ~either~ r_1=r_3,~ or~ r_1=\varepsilon.$$ One also has $r_1=u(r_0)=(r_0-r_3)(r_0p(r_0)-1)$.
\begin{enumerate}
    \item[(i)] If $r_1=r_3$, then 
    \begin{align*}
      r_1&=(r_0-r_3)(r_0p(r_0)-1)=(r_0-r_1)(r_0p(r_0)-1)\\
      &\implies r_1=r_0^2p(r_0)-r_0-r_0r_1p(r_0)+r_1\\
      &\implies r_0p(r_0)-1-r_1p(r_0)=0\\
      &\implies (r_0-r_1)p(r_0)=1\\
       &\implies r_1=r_0+\varepsilon 
    \end{align*} 
    It is easy to check that $u(x)=\varepsilon(r_1-x)(x+\varepsilon)$ satisfies (4.5). 

    \item[(ii)] If $r_1=\varepsilon$ and $r_1\neq r_3$, then a polynomial $u$ satisfying (4.5) must be in $N_{\varepsilon,1}$. Suppose $\varepsilon=1$. Clearly, $u(x)=(x-1)p(x)$ for some polynomial $p$. Since $u(r_0)=1$, one has $r_0=2$ and $p(2)=1$. Also, $r_3\neq 1$ and $u(r_3)=0$ implies that $p(r_3)=0$. Note that one also has $p(0)=-u(0)=-r_3$. Thus, $p(x)$ must be of the form $(x-r_3)(xq(x)+1)$ for some polynomial $q$. However, using the fact that $p(2)=1$, one immediately sees that $1=(2-r_3)(2q(2)+1)$, which is only possible if $r_3=3$. Then $u(x)=(x-1)^2(3-x)$ satisfies (4.5).
    Similarly, when $r_1=-1$, one can show that $r_0=-2,r_3=-3$, and that $u(x)=-(x+1)^2(x+3)$ satisfies (4.5). 
\end{enumerate}

\noindent{\tt Case 4.} Let $m=3$. Then the sequence looks like $r_0,r_1,r_2,0,r_4,0,\ldots.$ Then by Fact 1.4(b) it follows that $u(x)=(x-r_4)(xp(x)-1)$ for some polynomial $p$, and so $$0=u(r_2)=(r_2-r_4)(r_2p(r_2)-1)$$ implies, in particular, that $r_2=r_4$ or $r_2=\varepsilon$. Let us consider the truncated sequence $r_1,r_2,0,r_4,0,\ldots$. 

If $r_2=r_4$, then from Case 3(i) above it follows that $r_2=r_1+\varepsilon$. This means (4.5) must look like $r_0,r_1,r_1+\varepsilon,0,r_1+\varepsilon,0\ldots$. Suppose that $\varepsilon=1$. One immediately sees that $r_1-r_0|u(r_1)-u(r_2)=1$, i.e., $r_0=r_1\pm 1$. Since $r_2=r_1+1=(r_0\pm 1)+1$, and $r_2\neq r_0$, one has $r_1=r_0+1$ and $r_2=r_0+2$. Then $u(x)=(x-r_0-2)Q(x)$ for some polynomial $Q$ with $Q(0)=-1$. However, then $$r_0+1=r_1=u(r_0)=-2Q(r_0)\equiv_{r_0}-2Q(0)=2,$$ i.e., $r_0=1$ (as $r_0=-1$ means $r_1=0$). One sees that $u(x)=-x^3+4x^2-4x+3$ satisfies (4.5). Similarly, when $\varepsilon=-1$, one can show that $r_0=-1,r_1=-2,r_2=-3$, and taking $u(x)=-x^3-4x^2-4x-3$ suffices.

If $r_2=\varepsilon$, then one can see from Case 3(ii) above that $r_1=2\varepsilon,r_2=\varepsilon,r_4=3\varepsilon$, and $u$ satisfying the truncated sequence must be of the form $(x-\varepsilon)(x-3\varepsilon)(xq(x)+\varepsilon)$ for some polynomial $q$ with the property that $q(2\varepsilon)=-1$. Now, when we go back to the original sequence (4.5), it is clear that $r_0\not\in \{0,\varepsilon,2\varepsilon,3\varepsilon\}$. However, then $$2\varepsilon=u(r_0)=(r_0-\varepsilon)(r_0-3\varepsilon)(r_0q(r_0)+\varepsilon),$$ so that $(r_0-\varepsilon)(r_0-3\varepsilon)|2$, which is impossible unless $r_0=2\varepsilon$ and that possibility has been excluded above. \\

\noindent{\tt Case 5.} Let $m=4$. Then the sequence looks like $r_0,r_1,r_2,r_3,0,r_5,0,\ldots$. As the sequence of iterations of $v$ starting at $r_0$ is of the form (4.4), and $r_2\neq 0$, it follows from Theorem 4.9 that $r_0|2$ and $r_2=2r_1$. If $r_0=\varepsilon$, then $r_2=2\varepsilon$, and so it follows from either Fact 1.2 that no $u$ can satisfy (4.5). Therefore $r_0=2\varepsilon$. If $r_0=2$, then $r_2=4$. But the argument for $m=3$ applied to the polynomial sequence $r_1,r_2,r_3,0,r_5,0,\cdots$, implies that $r_2=\pm 1$, a contradiction. This proves that $m=4$ is not possible.

Now using the polynomial sequence $r_1,r_2,r_3,r_4,0,r_6,0,\ldots$ and the fact that $m=4$ is not possible (from Case 4), one readily sees that $m=5$ is not possible.\end{proof}

It should be noted that for Theorems 4.9 and 4.10, if we allow the generating polynomial $u$ to be a polynomial over $\mathbb{Q}$, then we can always use Lagrange's interpolation to obtain such a polynomial no matter how large the $m$ is, and so the restriction on $m$ that we get in the proof heavily depends on the fact that $u\in\mathbb{Z}[x]$. Therefore this is really a question about iterations of integer polynomials.

We end this paper by deriving an interesting result about bounded polynomial integer sequences which is a consequence of Theorems 4.9 and 4.10.

\begin{corollary}
Every bounded polynomial integer sequence must be of one of the following forms:
    \begin{enumerate}
        \item[(1)] $S,S,S,S,S,\ldots$

        \item[(2)] $R,S,S,S,S,\ldots$

        \item[(3)] $S+\varepsilon,S+2\varepsilon,S,S,S,\ldots$

         \item[(4)] $S+2\varepsilon,S+4\varepsilon,S,S,S,\ldots$

        \item[(5)] $S,R,S,R,S,\ldots$

        \item[(6)] $S+\varepsilon,S,R,S,R,\ldots$, where $R-S\neq \varepsilon$

        \item[(7)] $S+2\varepsilon,S+\varepsilon,S,S+3\varepsilon, S,\ldots$
    \end{enumerate}
for $\varepsilon\in\{\pm 1\}$, and $S,R$ integers.
    
\end{corollary}

\begin{proof}
    Let $\{r_n\}_{n\ge 0}$ be a bounded polynomial integer sequence. Then one obtains in particular, that there is an integer $S$ such that $r_n=S$ for infinitely many $n$'s, with $m:=\text{min}\{n\in\mathbb{N}\cup\{0\}~|~r_n=S\}$. Define $$s_n:=r_n-S~ (n\ge 0)\textup{, and }v(x):=u(x+S)-S.$$ Then $\{s_n\}_{n\ge 0}$ is a recurringly nilpotent polynomial sequence with $v$ satisfying $\{s_n\}_{n\ge 0}$, as for each $n\ge 0$ one obtains by induction that $v^{(n)}(x-S)=u^{(n)}(x)-S$.  Thus it follows from the proof Theorem 4.7 that $\{s_n\}_{n\ge 0}$ it must be of one of the forms (4.4) and (4.5). Thus from Theorems 4.9 and 4.10 one obtains the following list that contains all possible forms of bounded polynomial integer sequences.
    \begin{itemize}
        \item $m=0$, and the sequences are of the form
        \begin{align}
              & S,S,S,S,\ldots\\
              & S,r_1,S,r_1,\ldots;
        \end{align}
        
       \item $m=1$, and the sequences are of the form
        \begin{align}
              & r_0,S,S,S,\ldots\\
              & r_0,S,r_0,S,\ldots\\
              &S+\varepsilon,S,r_2,S,\ldots;
        \end{align}

        \item $m=2$, and the sequences are of the form
        \begin{align}
              & S+\varepsilon,S+2\varepsilon,S,S,\ldots\\
              & S+2\varepsilon,S+4\varepsilon,S,S,\ldots\\
              & r_0,r_0+\varepsilon,S,r_0+\varepsilon,\ldots\\
              & S+2\varepsilon,S+\varepsilon,S,S+3\varepsilon,S,\ldots;
        \end{align} 

        \item $m=3$, and the sequences are of the form 
        \begin{align}
            &S+\varepsilon,S+2\varepsilon,S+3\varepsilon,S,S+3\varepsilon;
        \end{align}

   \end{itemize}
    where $\varepsilon\in\{\pm 1\}$. One readily notices that (4.7) and (4.9) are of the form Corollary 4.11(2), and that (4.10) and (4.13) are of the form Corollary 4.11(4). Also, one can see that a sequence of the form (4.14) can be transformed into a sequence of the form (4.15) by simply replacing $S$ by $S-3\varepsilon$, and then replacing $\varepsilon$ by $-\varepsilon$. \end{proof}

We end this paper with a discussion on a few open problems. Given positive integers $r\text{ and }m$, $u\in N_{r,m}$, and $k\in \mathbb{N}$ satisfying the condition (2.3), we have, from Lemma 2.1, that $k\le C_r$. One notices that the bound on $k$ cannot be improved in general: for every given $k\ge 3$, and $r=k!-k-1$, consider $$u(x)=(x+1)-(x-r)\cdots(x-r-k+1).$$ This is a polynomial of degree $k$ satisfying (2.3), and $$u(r+k)=(r+k+1)-k!=0,~ i.e.,~ u^{(k+1)}(r)=0$$ In other words, $u(x)\in N_{r,k+1}^k$. Now suppose that $r$ is not in $\{s!-s-1~|~s\ge 3\}$. We ask the following two questions:
\begin{enumerate}
    \item[Q1.] Can the bound for $k$ be improved?

    \item[Q2.] Is it possible to get a universal bound for $k$ that does not depend on $r$?
\end{enumerate}
Also, as can be seen in second example in Section 1.2, whenever $m|r$ with $m\in\mathbb N$, $N_{r,m}$ is not empty. However, this is not necessary: the polynomial $u(x)=-x^3+12x^2-43x+46$ is nilpotent at $5$ of index 4. So we have a natural question.
\begin{enumerate}
    \item[Q3.] Given an $r\ge 5$, what are the positive integers $m$ for which the set $N_{r,m}$ is non-empty.
    
    \end{enumerate}

\noindent {\bf Acknowledgements.} The author gratefully acknowledges his advisor, Professor Alexander Borisov, for his invaluable suggestions and guidance with this paper, without which this work would not have seen the light of day. The author would like to thank
and acknowledge Professor Marcin Mazur for referring \cite{MN06}, \cite{N89} and \cite{N02}. Finally, the author would like to thank Professor Dikran Karagueuzian in catching a few
mistakes in the statements of Theorems 3.5, 4.3 and Corollary 4.4 in the
preliminary version of this manuscript.

\end{document}